\documentclass[10.5pt, a4paper]{amsart}
\usepackage{amssymb, amsmath, amscd, bm, amsthm, mathtools, pdfpages, setspace, subcaption, graphicx, xfrac, empheq, enumitem, array, arydshln, hyperref}

\usepackage[
backend=biber,
style=numeric-comp,
sortcites=true,
maxnames=99
]{biblatex}

\DeclareFieldFormat{pagetotal}{#1 pp}
\renewbibmacro*{note+pages}{
	\printfield{note}
	\setunit{\bibpagespunct}
	\printfield{pages}
	\setunit{\addcomma\space}
	\printfield{pagetotal}
	\clearfield{pagetotal}
}
\renewbibmacro*{chapter+pages}{
	\printfield{chapter}
	\setunit{\bibpagespunct}
	\printfield{pages}
	\setunit{\addcomma\space}
	\printfield{pagetotal}
	\clearfield{pagetotal}
}

\def\@Rref#1{\hbox{\rm \ref{#1}}}
\def\Rref#1{\@Rref{#1}}
\theoremstyle{plain}
\newtheorem{theorem}{Theorem}[section]
\newtheorem{proposition}[theorem]{Proposition}

\newtheorem{assumption}[theorem]{Assumption}
\newtheorem{lemma}[theorem]{Lemma}

\theoremstyle{definition}
\newtheorem{definition}{Definition}[section]

\newtheorem{remark}[definition]{Remark}

\newcommand{\smax}{\ensuremath{\mathrm{s}_{\max}}}
\newcommand{\smin}{\ensuremath{\mathrm{s}_{\min}}}
\newcommand{\Ls}{\ensuremath{\mathrm{L}^{2}}}
\newcommand{\Hs}{\ensuremath{\mathrm{H}^{1}_{0}}}
\newcommand{\Hso}{\ensuremath{\mathrm{H}^{1}}}

\newcommand{\sd}{\ensuremath{\mathrm{d}}}
\DeclareMathOperator{\Ran}{Ran}
\DeclareMathOperator{\Ker}{Ker}

\begin{document}

\title[Synthesis-operator approach]{Data-driven control of continuous-time 
systems: \\ A synthesis-operator approach}

\thispagestyle{plain}

\author{Masashi Wakaiki}
\address{Graduate School of System Informatics, Kobe University, Nada, Kobe, Hyogo 657-8501, Japan}
 \email{wakaiki@ruby.kobe-u.ac.jp}
 \thanks{This work was supported in part by 
 	JSPS KAKENHI Grant Number 24K06866.}

\begin{abstract}
This paper addresses data-driven control of 
continuous-time systems.
We develop a framework based on synthesis
operators associated with state and input trajectories.
A key advantage of the proposed method is that
it does not require the state derivative
and uses continuous-time data directly without sampling or filtering.
First, systems consistent with the data are 
represented in terms of synthesis operators, 
into which the data trajectories are embedded.
Next, we characterize data
informativity properties for
system identification and for stabilization
in the noise-free case.
Finally, we establish a necessary
and sufficient condition for noisy data to be informative
for quadratic stabilization.
All these informativity characterizations are formulated in terms of finite-dimensional matrices, by leveraging the finite-rank structure of the synthesis operators.
\end{abstract}


\maketitle

\section{Introduction}
\subsubsection*{Motivation and literature review}
Data-driven control has gained attention 
as an alternative paradigm 
for model-based control.
Instead of identifying an explicit system model, 
data-driven approaches 
analyze an unknown system and 
construct controllers 
directly from measured trajectory data.
One of the cornerstones in this field is
the fundamental lemma by 
Willems {\em et.~al.}~\cite{Willems2005}, which
essentially states that 
for linear time-invariant systems,
all possible trajectories 
can be represented by finitely many trajectories
if the input data are persistently exciting.
This lemma has spurred various
data-driven techniques for 
system analysis and 
controller design, using
system representations based on Hankel matrices of data trajectories; see, e.g., \cite{DePersis2020,Koch2022,Alsalti2025}.

The notion of data informativity,  introduced in \cite{Waarde2020},
provides another theoretical framework for data-driven control.
This framework examines whether the available data suffice to guarantee a specified property for all systems consistent with the data, 
and whether controllers that achieve the desired control objectives
for all data-consistent systems
can be constructed from the data.
Informativity for stabilization was characterized in \cite{Waarde2020}, and subsequent studies \cite{Waarde2022TAC,Waarde2023SIAM,Bisoffi2024,Kaminaga2025} addressed noisy data.
See the introductory
article \cite{Waarde2023} for further references on data informativity.
Much of the literature has focused on discrete-time systems, reflecting the sampled nature of measurements.
However, many physical systems evolve in continuous time.
This paper investigates informativity for stabilization of continuous-time systems.

For data-driven control of continuous-time systems,
it is assumed, e.g., in 
\cite{DePersis2020, Bisoffi2022, Eising2025} that 
the state derivative is either measured directly 
or estimated from sampled state trajectories.
In practice, however, 
accurate derivative information is often 
unavailable because measurements are corrupted by noise.
To overcome this limitation, 
several derivative-free data-driven control 
methods have been proposed.
In \cite{DePersis2024,Song2025}, discrete sequences
were obtained from the integral form of the state equation.
In \cite{Rapisarda2024}, the state and input trajectories were transformed
into discrete sequences using a polynomial orthogonal basis. 
The approach proposed in \cite{Ohta2024MTNS,Ohta2024,Wang2025}
is based on sampled data obtained via linear functionals.
Another line of work 
applied filtering to measured signals to 
avoid the state derivative; see
\cite{Bosso2025,Possieri2025,Bosso2025arXiv,Gao2025,Bosso2025Noisy,Possieri2026}.

\subsubsection*{Contributions and comparisons}
We develop a derivative-free approach 
for continuous-time systems 
by embedding data trajectories into 
synthesis operators.
Synthesis operators have been extensively 
investigated in 
frame theory
(see, e.g., \cite{Christensen2016}), and 
discrete synthesis operators were
employed to study 
data-driven control of discrete-time 
infinite-dimensional systems in \cite{Wakaiki2025Informativity}.
While
synthesis operators are typically defined on
$\Ls$-spaces,  in this work we consider 
synthesis operators acting on
$\Hs$-spaces. 
The main contributions of this synthesis-operator approach are as follows.
\begin{enumerate}
	\renewcommand{\labelenumi}{\textup{(\roman{enumi})}}
	\item We give a necessary and sufficient condition
	for a system to be consistent with given data, where
	synthesis operators 
	play the same role as data matrices in 
	the discrete-time setting.

	\item 
	We characterize data informativity properties for identification and for stabilization in 
	terms of synthesis operators under noise-free conditions. 
	The derived conditions can be verified through matrix computations by virtue of the finite-rank structure of the synthesis operators.
	\item We establish a necessary and sufficient condition
	for data corrupted by
	process noise to be informative for quadratic stabilization.
	This condition
	is in the form of linear matrix inequalities (LMIs) without
	approximation.
\end{enumerate}

The main advantage of 
this synthesis-operator approach, compared with the existing derivative-free methods mentioned above, 
is that it is formulated in terms of the original continuous-time data
without relying on sampled or filtered data representations.
Sampled data do not explicitly retain inter-sample information,
and filtering approaches require tuning parameters to condition the data.
Such reformulations may shift the focus of informativity analysis from the original data to the processed data, as in 
\cite{Rapisarda2024, Ohta2024MTNS,Ohta2024,Wang2025}.
In contrast, 
the synthesis-operator approach handles continuous-time data directly. This provides a framework to study not only sufficient but also necessary conditions for the original data to be informative, in parallel with the discrete-time setting.
Moreover, methods for
system analysis and controller design can be developed
without sampling or filtering parameters.

\subsubsection*{Organization}
In Section~\ref{sec:Data_syn},
we introduce the notion of synthesis operators and
use them to characterize
systems that are consistent with the given data.
Sections~\ref{sec:identification} and
\ref{sec:stabilization} investigate informativity properties 
of noise-free data
for system identification and for stabilization,
respectively.
In Section~\ref{sec:stabilization_noise},
we establish
a necessary and sufficient LMI condition
for noisy data to be informative for quadratic
stabilization. 
Section~\ref{sec:conclusion} provides concluding remarks.
\subsubsection*{Notation}
We denote by 
$\mathbb{S}^n$
the set of $n \times n $ real symmetric matrices
and by $I_n$
the $n \times n $  identity matrix.
The transpose 
and the Moore--Penrose 
pseudoinverse of a matrix $A$ are denoted by $A^{\top}$ and $A^+$, respectively.
If $A \in \mathbb{S}^n$ is positive definite (resp. nonnegative definite), then we write
$A \succ 0$ (resp. $A \succeq 0$).
Analogous notation applies to negative definite and nonpositive definite matrices.

Let $\tau >0$. 
We denote by $\Ls([0,\tau];\mathbb{R}^n)$
the space of measurable functions 
$f  \colon [0,\tau] \to \mathbb{R}^n$ satisfying
$\int_0^{\tau} \|f(t) \|^2 dt < \infty$.
The Sobolev space $\Hso([0,\tau];\mathbb{R}^n)$
consists of all absolutely continuous functions $\phi\colon
[0,\tau] \to \mathbb{R}^n$ satisfying $\phi' \in \Ls([0,\tau];\mathbb{R}^n)$.
We write
$\Ls[0,\tau] \coloneqq \Ls([0,\tau];\mathbb{R})$
and 
$\Hso[0,\tau] \coloneqq \Hso([0,\tau];\mathbb{R})$.
The space $\Hs[0,\tau]$ consists of those functions in $\Hso[0,\tau]$ that vanish at 
the endpoints $0$ and $\tau$.
The inner product on $\Ls[0,\tau]$ 
is 
\[
\langle f, g \rangle_{\Ls} = 
\int_0^\tau f(t) g(t)dt
\]
for $f,g \in \Ls[0,\tau]$, and 
the inner product on 
$\Hs[0,\tau]$
is 
\[
\langle \phi, \psi \rangle_{\Hs} = 
\int_0^\tau \phi'(t) \psi'(t)dt
\]
for $\phi,\psi \in \Hs[0,\tau]$.

Let $Y$ and $Z$ be Hilbert spaces.
We denote by $\mathcal{L}(Y,Z)$ the space of 
bounded linear operators from $Y$ to $Z$.
Let $T \in \mathcal{L}(Y,Z)$.
The range and kernel of $T$ are denoted by
$\Ran T$ and 
$\Ker T$, respectively.
The Hilbert space adjoint of $T$ is denoted by $T^*$.
The closure and the orthogonal complement of 
a subset $E$ of $Y$ are denoted by $\overline{E}$ and 
$E^{\perp}$, respectively.

\section{Preliminaries}
In this section, we introduce some preliminaries
for the proposed data-driven method.
First, we characterize the set of data-consistent systems 
in terms of synthesis operators.
Next, we derive integral representations for
products of synthesis operators and their adjoints.
These representations are crucial to
the numerical implementation of 
data-driven control based on synthesis operators.
Finally, we briefly review basic facts related to right inverses of
bounded operators.
\subsection{Data and synthesis operators}
\label{sec:Data_syn}
Fix $\tau >0$ and 
suppose that 
the state data $x \in \Hso([0,\tau]; \mathbb{R}^n)$ and the input data 
$u \in \Ls ([0,\tau]; \mathbb{R}^m)$ are available,
as in, for example, \cite{Song2025,Rapisarda2024,Ohta2024MTNS,Ohta2024,Wang2025,Bosso2025,Possieri2025}. 
These data are denoted by $(x,u)$.
To simplify the notation, we define
the data set $\Gamma_{\tau}$ and the system set 
$\Sigma_{n,m}$ by
\begin{align*}
	\Gamma_\tau &\coloneqq 
	\{
	(x,u): x \in \Hso([0,\tau]; \mathbb{R}^n)
	\text{~and~}
	u \in \Ls([0,\tau]; \mathbb{R}^m)
	\}, \\
	\Sigma_{n,m} &\coloneqq \{ 
	(A,B) : A \in \mathbb{R}^{n \times n}\text{~and~}
	B \in \mathbb{R}^{n \times m}
	\}.
\end{align*}
We assume that the data
$(x,u) \in \Gamma_{\tau}$ are generated by
some continuous-time linear system.
\begin{assumption}
	\label{assump:true_sys}
	There exists a system
	$(A_s,B_s) \in \Sigma_{n,m}$ such that
	for a.e.~$t \in [0,\tau]$,
	\begin{equation}
		\label{eq:true_system}
		x'(t) = A_sx(t) + B_su(t). 
	\end{equation}
\end{assumption}

The system 
$(A_s,B_s)$ can be regarded as
the true system generating the data $(x,u)$.
This paper considers the situation where 
the true system $(A_s,B_s)$ 
in Assumption~\ref{assump:true_sys} is unknown.

In the proposed approach, 
we embed the data into bounded linear operators on
$\Hs  [0,\tau]$.
These operators serve as continuous-time 
counterparts to
the matrices whose columns are the data vectors
in the discrete-time setting.
\begin{definition}
	\label{def:synthesis_op}
	Let $f \in \Ls ([0,\tau]; \mathbb{R}^n)$.
	Define
	the operators $F,F_{\sd} \in \mathcal{L}(\Hs  [0,\tau], \mathbb{R}^n)$  by
	\begin{align*}
		F\phi \coloneqq \int_0^\tau  \phi(t)f(t) dt
		\quad \text{and} \quad 
		F_{\sd}\phi  \coloneqq -\int_0^{\tau}  \phi '(t)f(t) dt 
	\end{align*}
	for $\phi  \in \Hs [0,\tau]$.
	We call $F$ the {\em synthesis operator associated with $f$}
	and  $F_{\sd}$ the {\em differentiated 
		synthesis operator associated with $f$}.
\end{definition}

Given data $(x,u) \in \Gamma_{\tau}$,
we denote by $X$ and $U$ the synthesis operators 
associated with $x$
and $u$, respectively.
We also 
denote by $X_{\sd}$ the differentiated synthesis operator
associated with $x$.
We call $(X_{\sd},X,U)$ the {\em synthesis operator triple
	associated with $(x,u)$}.

For data $\mathfrak{D} = (x,u)\in \Gamma_{\tau}$, we 
define the set $\Sigma_{\mathfrak{D}}$ of systems by
\begin{equation*}
	\Sigma_{\mathfrak{D}} \coloneqq 
	\{
	(A,B) \in \Sigma_{n,m}:
	x' = Ax +Bu~\text{a.e.~on $[0,\tau]$}
	\}.
\end{equation*}
The set $\Sigma_{\mathfrak{D}}$ consists of systems consistent with 
the data $\mathfrak{D}$.
In the following lemma, we characterize 
$\Sigma_{\mathfrak{D}}$ in terms of the synthesis operator triple
associated with the data.
\begin{lemma}
	\label{lem:synthesis}
	Let $(X_{\sd},X,U)$ be the synthesis operator triple
	associated with $\mathfrak{D} = (x,u) \in \Gamma_{\tau}$. 
	Then for all
	$(A,B) \in \Sigma_{n,m}$, one has 
	$(A,B) \in \Sigma_{\mathfrak{D}}$ if and only if
	$X_{\sd} = AX + BU$.
\end{lemma}
\begin{proof}
	Let $(A,B) \in \Sigma_{n,m}$
	and $\phi  \in \Hs [0,\tau]$.
	By integration by parts,
	$
	X_{\sd} \phi  = AX \phi  + BU \phi 
	$
	if and only if
	\[
	\int_0^\tau \phi (t) \big(x'(t) - Ax(t)- Bu(t) \big) dt = 0.
	\]
	The assertion follows
	from a standard property of test functions
	(see, e.g., \cite[Proposition~13.2.2]{Tucsnak2009}).
\end{proof}
\subsection{Products of synthesis operators and their adjoints}
Since the synthesis operators 
$F$ and $F_{\sd}$ introduced in 
Definition~\ref{def:synthesis_op} are of  finite rank,
products such as 
$F_{\sd}F_{\sd}^*$ and 
$F_{\sd} F^*$ can be identified with finite-dimensional matrices.
The proposed data-driven method
exploits this finite-dimensional property.
An important advantage of this approach is that
the resulting analysis and design procedures reduce to matrix-based computations despite the
underlying operator-theoretic framework.
The following lemma provides explicit formulae 
for products of synthesis operators and their adjoints, which are essential for
our synthesis-operator approach.
\begin{lemma}
	\label{lem:adjoint}
	Let $F$ and $F_{\sd}$ be
	the synthesis operator
	and 
	the differentiated 
	synthesis operator 
	associated with $f \in \Ls ([0,\tau];\mathbb{R}^n)$,
	respectively. 
	Define $\kappa\colon [0,\tau] \times [0,\tau] \to \mathbb{R}$ by
	\[
	\kappa(t,s) \coloneqq \begin{dcases}
		\frac{t(\tau-s)}{\tau},&t \leq s, \\
		\frac{s(\tau-t)}{\tau},&t > s.
	\end{dcases}
	\]
	Then
	the following statements hold:
	\begin{enumerate}
		\renewcommand{\labelenumi}{\textup{\alph{enumi})}}
		\item 
		The adjoint operators 
		$F^*,F_{\sd}^* \in \mathcal{L}(
		\mathbb{R}^n,\Hs [0,\tau]
		)$
		are given by
		\begin{align}
			\label{eq:F_ad}
			(F^* v)(t) &=
			\left(
			\int_0^\tau \kappa(t,s) f(s)  ds \right)^{\top}v,\\
			\label{eq:Fd_ad}
			(F_{\sd}^* v)(t) &=
			\left(  -\int_0^t f(s)  ds +
			\frac{t}{\tau}  \int_0^\tau f(s) ds \right)^{\top } v
		\end{align}
		for all $v \in \mathbb{R}^n$ and $t \in [0,\tau]$.
		\item The matrix 
		$F_{\sd}F_{\sd}^* \in \mathbb{R}^{n \times n}$ 
		is given by
		\begin{equation}
			\label{eq:Fsd_ad}
			F_{\sd} F_{\sd}^* = 
			\int_0^\tau f(t)f(t)^{\top} dt - 
			\frac{1}{\tau} \int_0^\tau f(t)dt
			\int_0^\tau f(t)^{\top}dt.
		\end{equation}
		\item
		Let $G$ be the synthesis operator
		associated with $g \in \Ls ([0,\tau];\mathbb{R}^m)$.
		Then the matrices 
		$GF^*, GF_{\sd}^* \in \mathbb{R}^{m \times n}$ 
		are given by
		\begin{align}
			\label{eq:GF_ad}
			GF^* &= 
			\int_0^\tau \int_0^\tau \kappa(t,s)g(t)
			f(s)^{\top} dsdt,\\
			\label{eq:GFd_ad}
			GF_{\sd}^* &= -
			\int_0^\tau \int_0^\tau \frac{\partial \kappa}{\partial s}(t,s)g(t)
			f(s)^{\top} dsdt.
		\end{align}
	\end{enumerate}
\end{lemma}
\begin{proof}
	a)
	Let $\phi \in \Hs [0,\tau]$ and $v \in \mathbb{R}^n$.
	Integration by parts yields
	\begin{equation*}
		\langle F\phi, v \rangle_{\mathbb{R}^n} = -
		\int_0^\tau \phi'(t) \int_0^t f(s)^{\top} vds dt,
	\end{equation*}
	where $\langle \cdot,\cdot\rangle_{\mathbb{R}^n}$
	denotes the standard inner product on
	the Euclidean space $\mathbb{R}^n$.
	By the definition of 
	adjoints, we also have
	\begin{equation*}
		\langle F\phi, v \rangle_{\mathbb{R}^n}  = 
		\langle \phi, F^*v \rangle_{\Hs}  =
		\int_0^\tau \phi'(t) (F^*v)'(t) dt.
	\end{equation*}
	Therefore,
	\begin{equation*}
		\int_0^\tau \phi'(t) 
		\left( (F^*v)'(t) + 
		\int_0^t f(s)^{\top} vds
		\right)
		dt =0 .
	\end{equation*}
	Since 
	\begin{align*}
		&\{
		\phi' : \phi \in \Hs [0,\tau]
		\}^{\perp} \\
		&\quad =
		\{
		\psi \in \Ls [0,\tau] : \psi(t) \equiv C
		\text{~for some $C \in \mathbb{R}$}
		\},
	\end{align*}
	there exists $C \in \mathbb{R}$
	such that
	\begin{equation}
		\label{eq:F_ad_with_C}
		(F^*v)'(t) = -
		\int_0^t f(s)^{\top} vds +C
	\end{equation}
	for a.e.~$t \in [0,\tau]$.
	The fundamental theorem of calculus and 
	Fubini's theorem show that 
	\begin{align*}
		(F^*v)(t) - (F^*v)(0) 
		&=
		-\int_0^t (t-s)f(s)^{\top}v ds + Ct
	\end{align*}
	for all $t \in [0,\tau]$.
	Recalling that $F^*v \in \Hs [0,\tau]$, we obtain
	\[
	C = 
	\frac{1}{\tau}\int_0^\tau (\tau-s)f(s)^{\top} vds.
	\]
	Since
	\[
	-(t-s)\chi_{[0,t]}(s) + \frac{t(\tau-s)}{\tau} = \kappa(t,s)
	\]
	for all $t,s \in [0,\tau]$,
	where $\chi_{[0,t]}$ denotes the characteristic function
	of the interval $[0,t]$,
	we conclude that
	\eqref{eq:F_ad} holds for all $t \in [0,\tau]$.
	The same argument shows that
	\begin{equation}
		\label{eq:Fd_deriv}
		(F_{\sd}^*v)'(t) = -
		f(t)^{\top} v +
		\frac{1}{\tau}
		\int_0^\tau f(s)^{\top} vds
	\end{equation}
	for a.e.~$t \in [0,\tau]$.
	Thus, \eqref{eq:Fd_ad} also holds for all $t \in [0,\tau]$.
	
	b)
	By \eqref{eq:Fd_deriv}, we obtain
	\begin{align*}
		F_{\sd} F_{\sd}^* v 
		&=
		-\int_0^\tau 
		\left(
		-f(t)^{\top}v +
		\frac{1}{\tau }
		\int_0^\tau f(s)^{\top} v ds
		\right)  f(t)dt 
	\end{align*}
	for all $v \in \mathbb{R}^n$.
	Therefore, \eqref{eq:Fsd_ad} holds.
	
	c)
	The first formula \eqref{eq:GF_ad} follows immediately
	from \eqref{eq:F_ad}. To prove
	the second formula \eqref{eq:GFd_ad},
	observe first that
	by \eqref{eq:Fd_ad},
	\begin{equation}
		\label{eq:FGv}
		GF_{\sd}^* v 
		=
		\int_0^\tau 
		\left( - \int_0^t f(s)^{\top } v ds +
		\frac{t}{\tau}  \int_0^\tau f(s)^{\top }v ds \right)  g(t)dt
	\end{equation}
	for all $v \in \mathbb{R}^n$.
	Using the characteristic function
	$\chi_{[0,t]}$, we obtain
	\[
	-\chi_{[0,t]}(s) + \frac{t}{\tau} = -\frac{\partial \kappa}{\partial s}(t,s)
	\]
	for all $t,s \in [0,\tau]$ with $t\neq s$.
	Hence, the right-hand integral of \eqref{eq:FGv}
	can be written as 
	\begin{equation*}
		- \left(\int_0^\tau 
		\int_0^\tau  \frac{\partial \kappa}{\partial s}(t,s) g(t)f(s)^{\top} ds
		dt \right) v.
	\end{equation*}
	Thus,
	we derive \eqref{eq:GFd_ad}.
\end{proof}

Note that by Lemma~\ref{lem:adjoint}.c),
we can also compute the matrix $F_{\sd} G^*$ by
using the relation $F_{\sd} G^* = (G F_{\sd}^*)^{\top}$.
In the synthesis-operator approach, the state-input data are mainly used to compute
the double integrals in
\eqref{eq:Fsd_ad}--\eqref{eq:GFd_ad}.
Even when only sampled data are available, the proposed method can still be approximately
implemented, provided that these integrals are accurately evaluated by numerical quadrature.

\begin{remark}
	Let $F$ and $F_{\sd}$ be the synthesis
	operator and the differentiated synthesis operator
	associated with the given data, respectively.
	In the modulating function
	method \cite{Shinbrot1954} for system identification,
	the focus is typically on selecting a specific sequence $(\phi_n)_{n \in \mathbb{N}}$
	of modulating functions
	in
	$\Hs[0,\tau]$, from which 
	the vectors $F\phi_n$ and $F_{\sd} \phi_n$ are formed;
	see also the survey~\cite{Preisig1993}.
	However, using only finitely many such vectors
	may fail to fully capture the information contained in the original data.
	In contrast, 
	the proposed data-driven approach directly treats
	synthesis operators themselves.
	The key to bridging this infinite-dimensional 
	framework with
	finite-dimensional computation lies in 
	the integral representations \eqref{eq:Fsd_ad}--\eqref{eq:GFd_ad} for
	products of synthesis operators and their adjoints. 
\end{remark}

\subsection{Right inverses}
\label{sec:right_inverse}
We conclude these preliminaries by 
recalling the definition and a parameterization of
right inverses.
Let $Y$ and $Z$ be Hilbert spaces.
\begin{definition}
	An operator $S \in \mathcal{L}(Z,Y)$ is called 
	a {\em right inverse of $T \in \mathcal{L}(Y,Z)$} if $TS = I$.
\end{definition}

By definition, the condition 
$\Ran T = Z$ is necessary for 
$T \in \mathcal{L}(Y,Z)$
to have a right inverse.
This surjectivity condition
is equivalent to the invertibility of $TT^*$;
see, e.g., \cite[Proposition~12.1.3]{Tucsnak2009}
for the proof.
\begin{lemma}
	\label{lem:surjectivity}
	Let $T \in \mathcal{L}(Y,Z)$. Then
	$\Ran T = Z$ if and only if $TT^*$ is invertible
	in $\mathcal{L}(Z)$.
\end{lemma}

This equivalence is computationally 
useful in our setting, since
products of synthesis operators and their 
adjoints can be identified with matrices, and their integral representations have been
obtained in Lemma~\ref{lem:adjoint}.

Conversely,
suppose that $T \in \mathcal{L}(Y,Z)$ is surjective. Then
there exists a right inverse of $T$, and 
every right inverse $S$  can be parameterized as
\begin{equation}
	\label{eq:right_inv_rep}
	S = T^*(TT^*)^{-1} + (I - T^*(TT^*)^{-1}T)S_0 \eqqcolon E(S_0)
\end{equation}
for some $S_0 \in \mathcal{L}(Z,Y)$.
Indeed, a direct calculation shows that 
$TE(S_0)= I$ for any $S_0 \in \mathcal{L}(Z,Y)$.
Moreover, if $S \in \mathcal{L}(Z,Y)$
satisfies $TS = I$, then $S$ can be written as 
$S = E(S)$.

\section{Data informativity for system identification}
\label{sec:identification}
In this section, we investigate the following data property
for system identification, 
following the discrete-time definition provided in
\cite[Definition~5]{Waarde2020}.
\begin{definition}
	Under
	Assumption~\ref{assump:true_sys},
	the data $\mathfrak{D} = (x,u)\in \Gamma_{\tau}$ are called {\em informative for
		system identification} if 
	$\Sigma_{\mathfrak{D}} = 
	\{(A_s,B_s) \}$.
\end{definition}

The following result
provides a characterization of  informativity for system identification, which 
is a continuous-time analogue of 
\cite[Proposition~6]{Waarde2020}.
Note that the matrices in statement~(iii) and \eqref{eq:sys_identification}
can be computed using  Lemma~\ref{lem:adjoint}.
\begin{proposition}
	\label{prop:identification}
	Suppose that the data $(x,u) \in \Gamma_{\tau}$ satisfy 
	Assumption~\ref{assump:true_sys}.
	Let $(X_{\sd},X,U)$ be the synthesis operator triple
	associated with $(x,u)$. 
	Then
	the following statements are equivalent:
	\begin{enumerate}
		\renewcommand{\labelenumi}{\textup{(\roman{enumi})}}
		\item The data $(x,u)$ are informative for
		system identification. \vspace{3pt}
		\item $\Ran \begin{bmatrix}
			X \\ U
		\end{bmatrix} = \mathbb{R}^{n + m}$. \vspace{3pt}
		\item 
		The matrix $\begin{bmatrix}
			X X^* & X U^* \\
			U X^* & U U^*
		\end{bmatrix} \in \mathbb{R}^{(n+m) \times (n+m)}$ is invertible. \vspace{3pt}
	\end{enumerate}
	Furthermore, if statement~(iii) holds, then
	the true system $(A_s,B_s)$ is given by
	\begin{equation}
		\label{eq:sys_identification}
		\begin{bmatrix}
			A_s & B_s
		\end{bmatrix} =
		\begin{bmatrix}
			X_{\sd} X^* & X_{\sd} U^* \\
		\end{bmatrix}
		\begin{bmatrix}
			X X^* & X U^* \\
			U X^* & U U^*
		\end{bmatrix}^{-1}.
	\end{equation}
\end{proposition}

\begin{proof}
	To simplify the notation, we write
	\[
	\mathfrak{D} \coloneqq (x,u) \quad \text{and} \quad
	T \coloneqq 
	\begin{bmatrix}
		X \\ U
	\end{bmatrix}.
	\]
	First, we prove the implication (i) $\Rightarrow $ (ii).
	Assume, for contradiction, that 
	$\Ran T \neq \mathbb{R}^{n + m}$.
	Then there exist $\xi \in \mathbb{R}^n$
	and $\upsilon \in \mathbb{R}^m$ such that 
	\[
	\begin{bmatrix}
		\xi \\ \upsilon 
	\end{bmatrix} \in 
	\left(\Ran T\right)^{\perp} \setminus \{ 0\}.
	\]
	Let $\zeta \in \mathbb{R}^n \setminus \{ 0\}$ be
	arbitrary, and define
	$(A_0,B_0) \in \Sigma_{n,m}$ by
	\begin{equation}
		\label{eq:A0B0_def}
		A_0 \coloneqq \zeta \xi^{\top}
		\quad \text{and} \quad 
		B_0 \coloneqq \zeta \upsilon^{\top}.
	\end{equation}
	Then $(A_s+A_0,B_s+B_0) \neq (A_s,B_s)$.
	Since
	\[
	A_0X \phi + B_0 U \phi =
	\zeta \begin{bmatrix}
		\xi^{\top} & \upsilon^{\top}
	\end{bmatrix}
	T \phi 
	=0
	\]
	for all $\phi \in \Hs[0,\tau]$,
	we obtain 
	$(A_s+A_0,B_s+B_0) \in \Sigma_{\mathfrak{D}}$ 
	by
	Lemma~\ref{lem:synthesis}.
	This is a contradiction.

	Next, we prove the implication (ii) $\Rightarrow $ (i).
	Let $(A,B) \in \Sigma_{\mathfrak{D}}$.
	Since $T$ is surjective,
	there exists a right inverse 
	$S$ 
	of $T$.
	Therefore, Lemma~\ref{lem:synthesis} shows that
	\begin{equation}
		\label{eq:AsBs_rep}
		\begin{bmatrix}
			A & B
		\end{bmatrix}
		=
		\begin{bmatrix}
			A & B
		\end{bmatrix}
		TS =
		X_{\sd} S.
	\end{equation}
	This implies that $(A,B)$ is uniquely determined by the 
	data $\mathfrak{D} $.
	Hence,
	$\mathfrak{D} $ are informative for
	system identification.
	
	The equivalence of
	(ii) and (iii) is an straightforward consequence of
	Lemma~\ref{lem:surjectivity}.
	The assertion 
	\eqref{eq:sys_identification}
	follows from \eqref{eq:right_inv_rep} with $S_0 = 0$ and \eqref{eq:AsBs_rep}.
\end{proof}

\section{Data informativity for stabilization}
\label{sec:stabilization}
We introduce a data property ensuring that all data-consistent systems can be stabilized by a common feedback gain. This property was originally
introduced for
discrete-time systems in \cite[Definition~12]{Waarde2020}.
\begin{definition}
	\label{def:informative_stabilization}
	The data $\mathfrak{D} = (x,u) \in \Gamma_{\tau}$ are called {\em informative for
		stabilization} if 
	there exists $K \in \mathbb{R}^{m \times n}$ such that
	$A+BK$ is Hurwitz for all
	$(A,B) \in \Sigma_{\mathfrak{D}}$.
\end{definition}

The aim of this section is to characterize informativity for stabilization.
\begin{theorem}
	\label{thm:informativity_stabilization}
	Suppose that the data $(x,u) \in \Gamma_{\tau}$ satisfy 
	Assumption~\ref{assump:true_sys}.
	Let $(X_{\sd},X,U)$ be the synthesis operator triple
	associated with $(x,u)$. 
	Then the following statements are equivalent:
	\begin{enumerate}
		\renewcommand{\labelenumi}{\textup{(\roman{enumi})}}
		\item The data $(x,u)$ are informative for
		stabilization.
		\item $\Ran X = \mathbb{R}^n$, 
		and there exists 
		a right inverse $S$
		of $X$  such that the matrix
		$X_{\sd} S
		\in \mathbb{R}^{n \times n}$ is Hurwitz. 
		\item 
		The matrix 
		$X X^* \in \mathbb{R}^{n\times n}$
		is invertible, and there exists a matrix $H \in \mathbb{R}^{n\times n}$ such that the matrix
		$F+GH$ is Hurwitz, where $F,G \in \mathbb{R}^{n\times n}$ are defined by
		\begin{align*}
			F &\coloneqq X_{\sd} X^* (X X^*)^{-1}, \\
			G &\coloneqq X_{\sd} X_{\sd}^* - 
			X_{\sd} X^* (X X^*)^{-1} X X_{\sd}^*.
		\end{align*}
	\end{enumerate}
	Furthermore, if statement (iii) holds, then the feedback gain
	\begin{equation}
		\label{eq:stabilizing_gain}
		K \coloneqq U X^* (X X^*)^{-1}
		+ (U X_{\sd}^* - U X^*(X X^*)^{-1} X X_{\sd}^*) H
	\end{equation}
	is such that $A+BK$ is Hurwitz for all $(A,B) \in \Sigma_{\mathfrak{D}}$.
\end{theorem}

The discrete-time counterpart of 
the equivalence of (i) and (ii) 
was established in
\cite[Theorem~16]{Waarde2020}.
In the continuous-time setting, we have
introduced the additional statement (iii) and 
the related controller design \eqref{eq:stabilizing_gain}, where
all the matrices are explicitly obtained by applying Lemma~\ref{lem:adjoint}.
An LMI-based characterization
of informativity for stabilization 
will also be provided; see
Proposition~\ref{prop:S_vs_QS} 
and Theorem~\ref{thm:noisy_case}
for details.

The following lemma is a continuous-time analogue of 
\cite[Lemma~15]{Waarde2020}; see also
\cite[Lemma~1]{Rapisarda2024}.
This result will be used in the proof of 
the equivalence of  (i) and (ii) in Theorem~\ref{thm:informativity_stabilization}.
For a square matrix $M$, 
we denote by $\smax(M)$ and $\smin(M)$
the maximum and minimum of the real
parts of the eigenvalues of $M$, respectively.
\begin{lemma}
	\label{lem:Ran_cond}
	Suppose that the data $\mathfrak{D} = (x,u) \in \Gamma_{\tau}$ satisfy Assumption~\ref{assump:true_sys}.
	Let $X$ and $U$ be the synthesis operators
	associated with $x$ and $u$, respectively.
	If $K \in \mathbb{R}^{m \times n}$ 
	satisfies 
	\begin{equation}
		\label{eq:A_BK_spb}
		\sup_{(A,B) \in \Sigma_{\mathfrak{D}}}\smax(A+BK) < \infty,
	\end{equation}
	then 
	\[
	\Ran 
	\begin{bmatrix}
		I_n \\ K
	\end{bmatrix}
	\subseteq 
	\Ran 
	\begin{bmatrix}
		X \\ U
	\end{bmatrix}.
	\]
\end{lemma}
\begin{proof}
	Define the set $\Sigma_{\mathfrak{D}}^0$ of systems by
	\[
	\Sigma_{\mathfrak{D}}^0 \coloneqq 
	\{
	(A_0,B_0) \in \Sigma_{n,m}  :
	0 = A_0x+B_0u~\text{a.e.~on $[0,\tau]$}
	\}.
	\]
	By the same argument as in
	Lemma~\ref{lem:synthesis}, we have
	\begin{equation}
		\label{eq:SigmaD0_equiv}
		(A_0,B_0) \in \Sigma_{\mathfrak{D}}^0 \quad 
		\Leftrightarrow \quad 
		0 = A_0X+B_0U.
	\end{equation}
	
	We will prove that 
	\begin{equation}
		\label{eq:A0B0_prop}
		A_0+B_0K = 0\quad
		\text{for all $(A_0,B_0) \in \Sigma_{\mathfrak{D}}^0$.}
	\end{equation}
	Let $(A,B) \in \Sigma_{\mathfrak{D}}$ and $(A_0,B_0) \in
	\Sigma_{\mathfrak{D}}^0$.
	Define 
	\[
	\Psi \coloneqq A+BK\quad \text{and} \quad 
	\Psi_0 \coloneqq 
	A_0+B_0K.
	\]
	For all $\rho \in \mathbb{R}$,
	\[
	(A+\rho A_0,B+\rho B_0) \in 
	\Sigma_{\mathfrak{D}}.
	\] 
	Hence, by \eqref{eq:A_BK_spb},
	there exists $C >0 $ such that 
	for all $\rho \in \mathbb{R}$,
	\[
	\smax(\Psi+\rho \Psi_0) \leq C.
	\]
	This gives
	\[
	\smax(\Psi/\rho+\Psi_0) \leq C / \rho
	\quad  \text{if $\rho >0$}.
	\]
	Letting $\rho \to \infty$, we obtain 
	$\smax (\Psi_0) \leq 0$.
	Similarly, since
	\[
	\smin(\Psi/\rho+\Psi_0) \geq C/\rho
	\quad  \text{if $\rho <0$},
	\]
	we obtain 
	$\smin (\Psi_0) \geq 0$.
	Hence, all eigenvalues of $\Psi_0$ lie on 
	the imaginary axis $i \mathbb{R}$.
	Noting that 
	\[
	(\Psi_0^{\top}A_0,\Psi_0^{\top}B_0) \in \Sigma_{\mathfrak{D}}^0,
	\] 
	we see that the eigenvalues of 
	$\Psi_0^{\top}\Psi_0$ are also on $i \mathbb{R}$.
	Since $\Psi_0^{\top}\Psi_0$  is symmetric,
	its spectral radius is zero, which implies that
	$\Psi_0 = 0$.
	Thus, \eqref{eq:A0B0_prop} is proved.

	Next, we will prove that 
	\begin{equation}
		\label{eq:ker_cond}
		\Ker 
		\begin{bmatrix}
			X^* & U^* 
		\end{bmatrix}
		\subseteq 
		\Ker 
		\begin{bmatrix}
			I_n & K^{\top} 
		\end{bmatrix}.
	\end{equation}
	Take
	\[
	\begin{bmatrix}
		\xi \\ \upsilon
	\end{bmatrix} \in 	\Ker 
	\begin{bmatrix}
		X^* & U^* 
	\end{bmatrix},
	\]
	and 
	let $\zeta \in \mathbb{R}^n \setminus \{ 0\}$.
	Using these vectors $\xi$, $\upsilon$, 
	and $\zeta$, we define
	$(A_0,B_0) \in \Sigma_{n,m}$ as in \eqref{eq:A0B0_def}.
	Since
	$
	X^* \xi + U^* \upsilon = 0,
	$
	it follows that for all $\phi \in \Hs[0,\tau]$,
	\[
	A_0 X \phi + B_0 U \phi  =
	\langle X^* \xi + U ^*\upsilon,  \phi \rangle_{\Hs}
	\zeta = 0. 
	\]
	Combining 
	this
	with \eqref{eq:SigmaD0_equiv} and \eqref{eq:A0B0_prop},
	we have $A_0+B_0K = 0$.
	Then 
	\[
	0=\zeta^{\top}(
	\zeta \xi^{\top} + \zeta \upsilon^{\top} K) =
	\|\zeta\|^2
	\left( 
	\begin{bmatrix}
		I_n & K^{\top}
	\end{bmatrix}
	\begin{bmatrix}
		\xi \\ \upsilon
	\end{bmatrix}
	\right)^{\top}.
	\]
	Since $\zeta \neq 0$, it follows that
	\[
	\begin{bmatrix}
		\xi \\ \upsilon
	\end{bmatrix} \in 	\Ker 
	\begin{bmatrix}
		I_n & K^{\top} 
	\end{bmatrix}.
	\]
	Hence, the inclusion \eqref{eq:ker_cond} holds.
	
	By \eqref{eq:ker_cond}, we have
	\begin{align*}
		\Ran 
		\begin{bmatrix}
			I_n \\ K
		\end{bmatrix} &=
		\left( \Ker \begin{bmatrix}
			I_n & K^{\top} 
		\end{bmatrix} \right)^{\perp} \subseteq 
		\left( 	\Ker 
		\begin{bmatrix}
			X^* & U^* 
		\end{bmatrix} \right)^{\perp} =
		\overline{
			\Ran 
			\begin{bmatrix}
				X \\ U
		\end{bmatrix}}.
	\end{align*}
	Since 
	$
	\begin{bmatrix}
		X \\ U
	\end{bmatrix}
	$
	is a finite-rank operator, its range is closed.
	Thus, the assertion is proved.
\end{proof}

We are now ready to prove the main result of this section.

\noindent\hspace{1em}{\textit{Proof of Theorem~\ref{thm:informativity_stabilization}:}}
Let $\mathfrak{D} \coloneqq (x,u)$.
First, we prove the implication (i) $\Rightarrow$ (ii).
Let $K \in \mathbb{R}^{m
	\times n}$ be such that
$A+BK$ is Hurwitz for all
$(A,B) \in \Sigma_{\mathfrak{D}}$.
Using  Lemma~\ref{lem:Ran_cond},
we obtain $\Ran X = \mathbb{R}^n$.
Moreover,
\[
\begin{bmatrix}
	I_n \\ K
\end{bmatrix}
=
\begin{bmatrix}
	X \\ U
\end{bmatrix}S
\]
for some 
$S \in \mathcal{L}(\mathbb{R}^n,\Hs [0,\tau])$
by Douglas' lemma; see \cite[Theorem~1]{Douglas1966}
and \cite[Proposition~12.1.2]{Tucsnak2009}.
Then
$S$ is a right inverse of $X$, and $K = U S$.
Since Lemma~\ref{lem:synthesis}
yields
$
X_{\sd} S  = A+BK
$
for all $(A,B) \in \Sigma_{\mathfrak{D}}$, 
it follows that $X_{\sd} S$ is also Hurwitz.

Next, we prove the implication (ii) $\Rightarrow$ (i).
Suppose that $X_{\sd}S$ 
is Hurwitz for some right inverse $S$ of $X$.
Define
$K \coloneqq U S$.
By Lemma~\ref{lem:synthesis},
we have
\begin{equation}
	\label{eq:ABK_Xid}
	A+BK = X_{\sd} S
\end{equation}
for all $(A,B) \in \Sigma_{\mathfrak{D}}$.
Therefore, the data $(x,u)$ are informative for stabilization.

For the proof of the equivalence of (ii) and (iii),
observe that
$\Ran X = \mathbb{R}^n$ if and only if
$X X^* $ 	
is invertible; see Lemma~\ref{lem:surjectivity}.
Therefore, it suffices to prove the equivalence of
the latter assertions under the assumption that 
$X X^*$ is invertible.
To this end, we 
define $\Pi \in \mathcal{L}(\Hs[0,\tau])$ by
\[
\Pi \coloneqq I - X^* (X X^*)^{-1} X.
\]
Then 
for all $S_0 \in \mathcal{L}(\mathbb{R}^n,\Hs[0,\tau])$, 
the right inverse $E(S_0)$ of $X$, 
defined as in \eqref{eq:right_inv_rep}, satisfies
\begin{equation}
	\label{eq:closed_loop_matrix_E}
	X_{\sd} E(S_0) = 
	F + X_{\sd} \Pi S_0.
\end{equation}

To show the implication (ii)
$\Rightarrow $ (iii), 
we suppose that there exists
a right inverse $S$
of $X$ such that
$X_{\sd} S$ is Hurwitz.
Set $S_0 \coloneqq S$. 
Since $S = E(S_0)$,
we deduce from \eqref{eq:closed_loop_matrix_E}
that
$F + X_{\sd} \Pi S_0$ is Hurwitz.
Using $\Pi = \Pi^*$ and $\Pi^2 = \Pi$,
we obtain
\[
\Ker (\Pi^* X_{\sd}^*) = \Ker (X_{\sd} \Pi \Pi^* X_{\sd}^*) = 
\Ker (X_{\sd} \Pi  X_{\sd}^*).
\]
Therefore, 
\begin{align}
	\overline{\Ran (X_{\sd} \Pi)} &= \big(\!\Ker (\Pi^* X_{\sd}^*) \big)^{\perp} = \big(\!\Ker (X_{\sd} \Pi  X_{\sd}^*)\big)^{\perp}  = \overline{\Ran (X_{\sd} \Pi  X_{\sd}^*)}.
	\label{eq:ranges_XiPi}
\end{align}
Since $X_{\sd}\Pi $ and $X_{\sd} \Pi  X_{\sd}^*$ are finite-rank
operators, their ranges are closed. This and \eqref{eq:ranges_XiPi} imply that
\begin{equation}
	\label{eq:Xi_Pi_range}
	\Ran (X_{\sd} \Pi) = \Ran (X_{\sd} \Pi X_{\sd}^*).
\end{equation}

For $k=1,\dots,n$, let $e_k$ be the $k$-th unit vector in $\mathbb{R}^n$.
By \eqref{eq:Xi_Pi_range},
for each $k=1,\dots,n$,
there exists $h_k \in \mathbb{R}^n$ such that 
\[
X_{\sd} \Pi S_0e_k = X_{\sd} \Pi X_{\sd}^* h_k.
\]
Define $H \in \mathbb{R}^{n\times n}$ by
$
H = \begin{bmatrix}
	h_1 & \cdots & h_n
\end{bmatrix}.
$
Then 
\[
X_{\sd }\Pi S_0 = X_{\sd} \Pi X_{\sd}^* H. 
\]
Since $X_{\sd} \Pi X_{\sd}^* = G$
by the definitions of $G$ and $\Pi$, we obtain
\begin{equation}
	\label{eq:FXiPiS}
	F + X_{\sd} \Pi S_0 = 
	F + GH.
\end{equation}
Therefore, $F+GH$
is Hurwitz.

Conversely, suppose that there exists
$H \in \mathbb{R}^{n \times n}$ such that
$F+GH$ is Hurwitz. Define
$S_0 \coloneqq X_{\sd}^* H$. From \eqref{eq:closed_loop_matrix_E} and
\eqref{eq:FXiPiS}, we deduce that 
$X_{\sd}S$ is Hurwitz
for the right inverse
$S = E(S_0)$. Moreover,
if we define $K \coloneqq U E(S_0)$, then
$A+BK$
is Hurwitz for all $(A,B) \in \Sigma_{\mathfrak{D}}$
by 
\eqref{eq:ABK_Xid}.
Since 
\[
U E(S_0) = U X^* (X X^*)^{-1}
+ (U X_{\sd}^* - U X^*(X X^*)^{-1} X X_{\sd}^*) H,
\]
the remaining assertion concerning the 
controller design follows.
\hspace*{\fill} $\blacksquare$

\section{Informativity of noisy data for quadratic stabilization}
\label{sec:stabilization_noise}

This section addresses the design of stabilizing gains based on
data corrupted by process noise.
First, we introduce the class of noise considered in this paper, 
and then
present a necessary and sufficient condition for noisy data to be informative for quadratic stabilization. Finally, we illustrate the result with a numerical example.

\subsection{Noisy data}
\label{sec:noisy_case}

In addition to the state-input data $(x,u) \in \Gamma_{\tau}$,
we consider the process noise $w \in \Ls([0,\tau];\mathbb{R}^n)$.
Throughout this section, 
we denote by $W$ the synthesis operator associated with 
$w$.
Let $(X_{\sd},X, U)$ 
be the synthesis operator triple associated with $(x,u)$.
By Lemma~\ref{lem:synthesis},
we obtain the following equivalence for
a fixed $(A,B) \in \Sigma_{n,m}$:
\begin{equation}
	\label{eq:equiv_noise}
	x'= Ax +Bu+w
	~\text{a.e.~on $[0,\tau]$}
	~~ \Leftrightarrow ~~
	X_{\sd} = AX + BU+W.
\end{equation}

By \eqref{eq:equiv_noise},
the effect of the noise $w$ on the data can be 
evaluated using the operator norm $\|W\|$ rather than
the signal norm $\|w\|_{\Ls}$.
For $c \geq 0$, noting that $\|W\| \leq \sqrt{c}$ is equivalent to
$WW^* \preceq c I_n$, 
we define 
the noise class $\Delta_{\tau,c}$ by
\[
\Delta_{\tau,c} \coloneqq \{
w \in \Ls([0,\tau];\mathbb{R}^n ): 
WW^*\preceq  c I_n 
\}.
\]
See
Section~\ref{sec:norm_adjoint}
for a discussion of the operator norm $\|W\|$.

We assume that 
the 
data $(x,u) \in \Gamma_{\tau}$ are generated 
by some continuous-time linear 
system subject to process noise in 
$\Delta_{\tau,c}$.
\begin{assumption}
	\label{assump:true_sys_noise}
	There exist a system
	$(A_s, B_s) \in  \Sigma_{n,m}$ and a noise signal 
	$w_s \in \Delta_{\tau,c}$ such that
	for a.e.~$t \in [0,\tau]$,
	\begin{equation}
		\label{eq:true_sys_noise}
		x'(t) = A_sx(t) + B_su(t) + w_s(t).
	\end{equation}
\end{assumption}

The state-input data $\mathfrak{D} = (x,u)$ and
the noise-intensity parameter $c$ 
are available, but
we consider the situation where
the system $(A_s,B_s)$ and the noise $w_s$ in
Assumption~\ref{assump:true_sys_noise} are unknown.
Given data $\mathfrak{D}$ and a constant $c \geq 0$,
we define the set $\Sigma_{\mathfrak{D},c}$ of systems by
\begin{align*}
	\Sigma_{\mathfrak{D},c} \coloneqq 
	\{
	(A,B) \in \Sigma_{n,m} : 
	\text{there exists $w \in \Delta_{\tau,c}$
		such that~} x' = Ax +Bu+w
	~~\text{a.e.~on $[0,\tau]$}
	\}.
\end{align*}
Systems in $\Sigma_{\mathfrak{D},c}$ are consistent with
the data $\mathfrak{D}$ corrupted by process noise
in $\Delta_{\tau,c}$.
We have
$(A_s,B_s) \in \Sigma_{\mathfrak{D},c}$
under Assumption~\ref{assump:true_sys_noise}.

We introduce the Lyapunov-based 
notion of informativity for stabilization as in the discrete-time case \cite[Definition~3]{Waarde2022TAC}.
\begin{definition}
	\label{def:informative_QS}
	Let $c \geq 0$.
	The data $\mathfrak{D} = (x,u)\in \Gamma_{\tau}$ are called {\em informative for
		quadratic stabilization under the noise class $\Delta_{\tau,c}$} if 
	there exist $P \in \mathbb{S}^n$ and 
	$K \in \mathbb{R}^{m \times n}$ 
	such that $P \succ 0$ and
	\begin{equation}
		\label{eq:Lyapunov_ineq}
		-(A+BK)P-P(A+BK)^{\top} \succ 0
	\end{equation}
	for all $(A,B) \in \Sigma_{\mathfrak{D},c}$.
	In the case $c=0$, 
	the data $\mathfrak{D}$ are simply called {\em informative for
		quadratic stabilization}.	
\end{definition}

In Definition~\ref{def:informative_QS},
the matrices $P$ and $K$ are required to be common
to all $(A,B) \in \Sigma_{\mathfrak{D},c}$.
In contrast, in the definition of informativity for stabilization
(Definition~\ref{def:informative_stabilization}),
although the feedback gain $K$ remains common, the
positive definite 
matrix $P$ satisfying the Lyapunov inequality \eqref{eq:Lyapunov_ineq} 
may depend on the system $(A,B)$.
However, the following proposition shows that 
these informativity properties
are equivalent in the noise-free case.
\begin{proposition}
	\label{prop:S_vs_QS}
	Suppose that the data $(x,u)\in \Gamma_{\tau}$ satisfy 
	Assumption~\ref{assump:true_sys}.
	Then $(x,u)$ are informative for
	stabilization if and only if 
	$(x,u)$ are informative for
	quadratic stabilization.
\end{proposition}
\begin{proof}
	Suppose that 
	$\mathfrak{D} = (x,u)$ are informative for 
	stabilization.
	By Theorem~\ref{thm:informativity_stabilization},
	there exists
	a right inverse $S$ of $X$ such that the matrix
	$X_{\sd} S$ is Hurwitz.
	There exists $P \in \mathbb{S}^n$ 
	such that $P \succ 0$ and
	\[
	-(X_{\sd} S)P-P(X_{\sd} S)^{\top} \succ 0.
	\]
	On the other hand,
	if we define $K \coloneqq U S$, then
	$X_{\sd} S = A+BK$ 
	for all $(A,B) \in \Sigma_{\mathfrak{D}}$ by
	Lemma~\ref{lem:synthesis}.
	Hence,
	$(x,u)$ are informative for
	quadratic stabilization.
	The converse implication follows 
	immediately from Definitions~\ref{def:informative_stabilization}
	and \ref{def:informative_QS}.
\end{proof}

\subsection{Characterization via the  matrix S-lemma}
We now characterize the informativity 
of noisy data.
This characterization is a continuous-time analogue of 
\cite[Theorem~5.1]{Waarde2023SIAM}.
A similar result, up to truncation errors,
was obtained 
for discrete sequences transformed from
continuous-time data via polynomial
orthogonal bases in
\cite[Theorem~5]{Rapisarda2024}.
The following theorem gives a necessary and sufficient 
LMI condition for informativity without approximation. 
In this theorem, 
a central role is again played by products 
of the synthesis operators and their adjoints, 
whose formulae are provided in Lemma~\ref{lem:adjoint}.
\begin{theorem}
	\label{thm:noisy_case}
	Suppose that the data $\mathfrak{D}=(x,u) \in \Gamma_{\tau}$ and
	the noise-intensity parameter $c \geq 0$ satisfy Assumption~\ref{assump:true_sys_noise}.
	Let $(X_{\sd},X,U)$ be the synthesis operator triple
	associated with $(x,u)$. 
	Then the following statements are equivalent:
	\begin{enumerate}
		\renewcommand{\labelenumi}{\textup{(\roman{enumi})}}
		\item The data $(x,u)$ are informative for
		quadratic stabilization under the noise class $\Delta_{\tau,c}$.
		\item There exist
		matrices 
		$P \in \mathbb{S}^{n}$, 
		$L \in \mathbb{R}^{m \times n}$ 
		and a scalar $\alpha \geq 0$ such that
		$P\succ 0$ and 
		\begin{equation}
			\label{eq:LMI_noisy_case}
			\begin{bmatrix}
				\alpha (X_{\sd}X_{\sd}^* - cI_n ) - I_n& 
				-P -\alpha X_{\sd}X^* & -L^{\top} - \alpha X_{\sd}U^* \\
				-P - \alpha XX_{\sd}^* &  \alpha X X^* & 
				\alpha X U^* \\
				-L -\alpha U X_{\sd}^* & 
				\alpha U X^* & 
				\alpha U U^*
			\end{bmatrix}  \succeq 0.
		\end{equation}
	\end{enumerate}
	Furthermore, 
	if statement (ii) holds,
	then 
	$K \coloneqq LP^{-1}$ is such that 
	$A+BK$ is Hurwitz for all 
	$(A,B) \in \Sigma_{\mathfrak{D},c}$.
\end{theorem}

In the proof 
of Theorem~\ref{thm:noisy_case},
we employ 
the {\em matrix $S$-lemma} \cite[Corollary~4.13]{Waarde2023SIAM}.
To state this lemma,
we consider $\mathcal{M}, \mathcal{N} \in 
\mathbb{S}^{q+r}$ 
partitioned as 
\begin{equation}
	\label{eq:M_N_decomp}
	\mathcal{M} =
	\begin{bmatrix}
		\mathcal{M}_{11} & \mathcal{M}_{12} \\
		\mathcal{M}_{12}^{\top} & \mathcal{M}_{22}
	\end{bmatrix}\quad \text{and} \quad 
	\mathcal{N} =
	\begin{bmatrix}
		\mathcal{N}_{11} & \mathcal{N}_{12} \\
		\mathcal{N}_{12}^{\top} & \mathcal{N}_{22}
	\end{bmatrix},
\end{equation}
where 
the $(1,1)$-blocks have size $q \times q$ and 
the $(2,2)$-blocks have size $r \times r$.
Define
the matrix sets $\mathcal{Z}_{q,r}(\mathcal{N})$
and $\mathcal{Z}_{q,r}^+(\mathcal{M})$
by
\begin{align}
	\label{eq:ZN_def}
	\mathcal{Z}_{q,r}(\mathcal{N}) &\coloneqq 
	\left\{
	Z \in \mathbb{R}^{r\times q} :
	\begin{bmatrix}
		I_q \\ Z
	\end{bmatrix}^{\top}
	\mathcal{N}
	\begin{bmatrix}
		I_q \\ Z
	\end{bmatrix} \succeq 0
	\right\}, \\
	\label{eq:ZM_def}
	\mathcal{Z}_{q,r}^+(\mathcal{M}) &\coloneqq 
	\left\{
	Z \in \mathbb{R}^{r\times q} :
	\begin{bmatrix}
		I_q \\ Z
	\end{bmatrix}^{\top}
	\mathcal{M}
	\begin{bmatrix}
		I_q \\ Z
	\end{bmatrix}\succ 0
	\right\}.
\end{align}
We now state the matrix $S$-lemma.
\begin{lemma}
	\label{lem:S_lemma}
	Let $\mathcal{M},\mathcal{N} \in \mathbb{S}^{q+r}$
	be 
	partitioned as in \eqref{eq:M_N_decomp}.
	Assume that 
	\begin{equation}
		\label{eq:S_lem_assump1}
		\mathcal{M}_{22} \preceq 0,\quad \mathcal{N}_{22} \preceq 0, \quad 
		\Ker \mathcal{N}_{22} \subseteq \Ker \mathcal{N}_{12},
	\end{equation}
	and that 
	\begin{equation}
		\label{eq:S_lem_assump2}
		\mathcal{N}_{11} -
		\mathcal{N}_{12}\mathcal{N}_{22}^+\mathcal{N}_{12}^{\top} \succeq 0.
	\end{equation}
	Then the following statements are equivalent:
	\begin{enumerate}
		\renewcommand{\labelenumi}{\textup{(\roman{enumi})}}
		\item $\mathcal{Z}_{q,r}(\mathcal{N}) \subseteq 
		\mathcal{Z}_{q,r}^+(\mathcal{M})$. 
		\item
		There exist scalars
		$\alpha \geq 0$ and $\beta >0$ such that
		\begin{equation}
			\label{eq:MN_ineq}
			\mathcal{M} - \alpha \mathcal{N} \succeq
			\begin{bmatrix}
				\beta I_q & 0 \\ 0 & 0
			\end{bmatrix}.
		\end{equation}
	\end{enumerate}
\end{lemma}

Next, we present an auxiliary result on
the system set $\Sigma_{\mathfrak{D},c}$.
For a constant $c\geq 0$ and
operators $X_{\sd},X \in \mathcal{L}(\Hs[0,\tau],\mathbb{R}^n)$
and $U \in \mathcal{L}(\Hs[0,\tau],\mathbb{R}^m)$,
we define the matrix $\mathcal{N} \in \mathbb{S}^{2n+m}$ by
\begin{equation}
	\label{eq:N_def}
	\mathcal{N} \coloneqq 
	\begin{bmatrix}
		cI_n - X_{\sd} X_{\sd}^* & 
		X_{\sd} X^* & 
		X_{\sd} U^* \\
		X X_{\sd}^* & -X X^* & 
		-X U^* \\
		U X_{\sd}^* & 
		-U X^* & -U U^* 
	\end{bmatrix}
\end{equation}
The following lemma gives a characterization
of $\Sigma_{\mathfrak{D},c}$ in terms of $\mathcal{Z}_{n,n+m}(\mathcal{N})$.
\begin{lemma}
	\label{lem:Z_Sigma_relation}
	Let $(X_{\sd},X,U)$ be the synthesis operator triple
	associated with $\mathfrak{D} = (x,u) \in \Gamma_{\tau}$. 
	For a constant $c \geq 0$,
	define $\mathcal{N} \in \mathbb{S}^{2n+m}$ 
	by \eqref{eq:N_def}
	and $\mathcal{Z}_{n,n+m}(\mathcal{N})$ by \eqref{eq:ZN_def}.
	Then the following statements are equivalent 
	for all $(A, B) \in \Sigma_{n,m}$:
	\begin{enumerate}
		\renewcommand{\labelenumi}{\textup{(\roman{enumi})}}
		\item 
		$(A,B) \in \Sigma_{\mathfrak{D},c}$. \vspace{2pt}
		\item 
		$\begin{bmatrix}
			A & B
		\end{bmatrix}^{\top} \in \mathcal{Z}_{n,n+m}(\mathcal{N})
		$.\vspace{2pt}
	\end{enumerate}
\end{lemma}
\begin{proof}
	By the definition \eqref{eq:N_def} of $\mathcal{N}$,
	we have
	\begin{equation}
		\label{eq:N_c_Xid}
		\begin{bmatrix}
			I_n \!&\! A \!& \! B
		\end{bmatrix}\mathcal{N}
		\begin{bmatrix}
			I_n \\ A^{\top} \\ B^{\top}
		\end{bmatrix} \!=
		cI_n - (X_{\sd} - AX-BU)
		(X_{\sd} - AX-BU)^*
	\end{equation}
	for all $(A,B) \in \Sigma_{n,m}$.
	First, we prove the implication (i)
	$\Rightarrow$ (ii). Let
	$(A,B) \in \Sigma_{\mathfrak{D},c}$.
	By \eqref{eq:equiv_noise},
	there exists $w \in \Delta_{\tau,c}$ such that 
	the synthesis operator $W$ associated with
	$w$ satisfies 
	\[
	X_{\sd} - AX-BU = W.
	\]
	Substituting this into \eqref{eq:N_c_Xid} and
	using
	$ WW^*\preceq  c I_n $, we obtain
	$\begin{bmatrix}
		A & B
	\end{bmatrix}^{\top} \in \mathcal{Z}_{n,n+m}(\mathcal{N})
	$.
	
	Next, we prove the implication (ii)
	$\Rightarrow$ (i).
	Suppose that $\begin{bmatrix}
		A & B
	\end{bmatrix}^{\top} \in \mathcal{Z}_{n,n+m}(\mathcal{N})$.
	By
	\eqref{eq:N_c_Xid}, we have
	\begin{equation}
		\label{eq:c_Xid_ineq}
		cI_n - (X_{\sd} - AX-BU)
		(X_{\sd} - AX-BU)^* \succeq 0.
	\end{equation}
	Integrating by parts, we obtain
	\[
	X_{\sd}\phi - AX\phi - BU\phi =
	\int_0^\tau \phi(t) \big(x'(t)-Ax(t)-Bu(t)\big)dt
	\]
	for all $\phi \in \Hs[0,\tau]$.
	This implies that
	$X_{\sd} - AX-BU$
	is the synthesis operator associated with 
	\[
	w \coloneqq x'-Ax-Bu
	\in \Ls([0,\tau];\mathbb{R}^n).
	\] 
	By \eqref{eq:c_Xid_ineq}, we have
	$w \in \Delta_{\tau,c}$.
	Thus,
	$(A,B) \in 
	\Sigma_{\mathfrak{D},c}$ holds.
\end{proof}

We are now in a position to
prove Theorem~\ref{thm:noisy_case}.

\noindent\hspace{1em}{\textit{Proof of Theorem~\ref{thm:noisy_case}:}}
First, we prove the implication 
(i)
$\Rightarrow $ (ii). By assumption,
there exist $P \in \mathbb{S}^{n}$
and $K \in \mathbb{R}^{m \times n}$ 
such that $P \succ 0$ 
and the Lyapunov inequality \eqref{eq:Lyapunov_ineq} holds
for all $(A,B) \in \Sigma_{\mathfrak{D},c}$.
Define $\mathcal{M} \in \mathbb{S}^{2n+m}$ by
\begin{equation}
	\label{eq:M_def}
	\mathcal{M} \coloneqq 
	\begin{bmatrix}
		0 & -P & - PK^{\top} \\
		-P & 0 & 0 \\
		-KP & 0 & 0
	\end{bmatrix}.
\end{equation}
Let $\mathcal{N} \in \mathbb{S}^{2n+m}$ be defined as in \eqref{eq:N_def}.
By the Lyapunov inequality \eqref{eq:Lyapunov_ineq},
we obtain
$
\begin{bmatrix}
	A & B
\end{bmatrix}^{\top} \in 
\mathcal{Z}_{n,n+m}^+(\mathcal{M})
$
for all  $(A,B) \in \Sigma_{\mathfrak{D},c}$.
Therefore,
Lemma~\ref{lem:Z_Sigma_relation} yields
$\mathcal{Z}_{n,n+m}(\mathcal{N}) \subseteq
\mathcal{Z}_{n,n+m}^+(\mathcal{M})$.

To apply Lemma~\ref{lem:S_lemma} (the matrix $S$-lemma),
we will show that the conditions in
\eqref{eq:S_lem_assump1} and 
\eqref{eq:S_lem_assump2}
are satisfied. 
The $(2,2)$-block $\mathcal{M}_{22} \in \mathbb{S}^{n+m}$ of 
$\mathcal{M}$ is given by $\mathcal{M}_{22} = 0$, and
hence it is clear that
$ \mathcal{M}_{22}\preceq 0$.
The $(2,2)$-block $\mathcal{N}_{22}
\in \mathbb{S}^{n+m}$ of 
$\mathcal{N}$ satisfies
\[
\mathcal{N}_{22} = 
-\begin{bmatrix}
	X \\ U
\end{bmatrix}
\begin{bmatrix}
	X^* & U^*
\end{bmatrix} \preceq 0.
\]
Since
\[
\Ker  \mathcal{N}_{22} = 
\Ker \begin{bmatrix}
	X^*&  U^*
\end{bmatrix},
\]
the $(1,2)$-block $\mathcal{N}_{12} =
X_{\sd} 
\begin{bmatrix}
	X^*&  U^*
\end{bmatrix}$  of 
$\mathcal{N}$ 
satisfies $\Ker \mathcal{N}_{22} \subseteq
\Ker \mathcal{N}_{12}$. Therefore,
the conditions in \eqref{eq:S_lem_assump1} 
are satisfied.

To verify the condition in \eqref{eq:S_lem_assump2},
we define $\mathcal{F}_{\mathcal{N}}\colon
\mathbb{R}^{(n+m) \times n} \to \mathbb{R}^{n \times n}$ by
\[
\mathcal{F}_{\mathcal{N}}(Z) 
\coloneqq 
\begin{bmatrix}
	I_n \\ Z
\end{bmatrix}^{\top}
\mathcal{N} 
\begin{bmatrix}
	I_n \\ Z
\end{bmatrix}.
\]
Since $\Ker \mathcal{N}_{22} \subseteq
\Ker \mathcal{N}_{12}$,
it follows from \cite[Fact~8.9.7]{Bernstein2018} that 
\begin{align*}
	\mathcal{F}_{\mathcal{N}}(Z)
	&=
	\mathcal{N}_{11} -
	\mathcal{N}_{12}\mathcal{N}_{22}^+\mathcal{N}_{12}^{\top}   +
	(Z+ \mathcal{N}_{22}^+ \mathcal{N}_{12}^{\top})^{\top}
	\mathcal{N}_{22}
	(Z+ \mathcal{N}_{22}^+ \mathcal{N}_{12}^{\top})
\end{align*}
for all $Z \in \mathbb{R}^{(n+m) \times n}$.
By Assumption~\ref{assump:true_sys_noise}
and Lemma~\ref{lem:Z_Sigma_relation},
we obtain
$\begin{bmatrix}
	A_s & B_s
\end{bmatrix}^{\top} \in \mathcal{Z}_{n,n+m}(\mathcal{N})
$.
	Combining this with
	$\mathcal{N}_{22} \preceq 0$,
	we derive
	\[
	\mathcal{N}_{11} -
	\mathcal{N}_{12}\mathcal{N}_{22}^+\mathcal{N}_{12}^{\top} \succeq
	\mathcal{F}_{\mathcal{N}}
	\left(\begin{bmatrix}
		A_s & B_s
	\end{bmatrix}^{\top}
	\right) \succeq 0.
	\]

	By Lemma~\ref{lem:S_lemma} (the matrix $S$-lemma), there exist
	scalars $\alpha \geq 0$ and $\beta >0$ such that
	the inequality \eqref{eq:MN_ineq} 
	holds with $q=n$.
	Replacing 
	$P$,
	$KP$, and $\alpha$ by $\beta P$, $\beta L$, 
	and $\alpha \beta$, respectively,
	we conclude that the LMI~\eqref{eq:LMI_noisy_case}
	is satisfied.
	
	Next, we prove the implication 
	(ii)
	$\Rightarrow $ (i).
	Define $K \coloneqq LP^{-1}$ and let 
	$\mathcal{M},\mathcal{N} \in \mathbb{S}^{2n+m}$
	be as in \eqref{eq:M_def} and
	\eqref{eq:N_def}, respectively.
	Let $(A,B) \in \Sigma_{\mathfrak{D},c}$.
	By Lemma~\ref{lem:Z_Sigma_relation} and 
	the LMI~\eqref{eq:LMI_noisy_case}, we obtain
	\begin{align}
		0 &\preceq 
		\begin{bmatrix}
			I_n & A & B
		\end{bmatrix}
		\left(
		\mathcal{M} - \alpha
		\mathcal{N} - 
		\begin{bmatrix}
			I_n & 0 \\ 0 & 0
		\end{bmatrix}
		\right)
		\begin{bmatrix}
			I_n   \\ 
			A^{\top} \\
			B^{\top}
		\end{bmatrix}
		\notag \\
		&\preceq
		-(A+BK)P - P(A+BK)^{\top}
		-  I_n \notag \\
		&\prec
		-(A+BK)P - P(A+BK)^{\top}.
		\label{eq:Lyap_proof} 
	\end{align}
	Thus, the data $(x,u)$ are informative for
	quadratic stabilization under the noise class $\Delta_{\tau,c}$.
	The last assertion concerning the Hurwitz property of $A+BK$
	follows immediately from \eqref{eq:Lyap_proof}.
	\hspace*{\fill} $\blacksquare$
	
	\subsection{Numerical example}
	\label{sec:example}
	We consider the linearized model of a 
	batch reactor given in \cite[p.~213]{Rosenbrock1974}, which
	was also used as a numerical example 
	in the data-driven control literature, e.g., \cite{DePersis2020, Possieri2025,Bosso2025}.
	The true system $(A_s, B_s)$
	and the initial state $x_0$ are given by
	\begin{align*}
		A_s = 
		\begin{bmatrix}
			1.38 &-0.2077& 6.715& -5.676 \\
			-0.5814 &-4.29 &0 &0.675 \\
			1.067 &4.273 &-6.654 &5.893 \\
			0.048 &4.273 &1.343 &-2.104
		\end{bmatrix}, \quad 
		B_s = 
		\begin{bmatrix}
			0& 0\\ 5.679& 0\\ 1.136 &-3.146\\ 1.136 &0
		\end{bmatrix},\quad 
		x_0 = 
		\begin{bmatrix}
			1 \\ -1\\ 0\\ 1
		\end{bmatrix}.
	\end{align*}
	The matrix $A_s$ has two unstable eigenvalues, $1.9910$
	and $0.0635$.
	State and input trajectories are generated on the interval
	$[0,1]$, where the input $u$ is chosen as
	\[
	u(t) =5
	\begin{bmatrix}
		\sin (2\pi t) + \sin (4\pi t) \\
		\sin (3\pi t) + \sin (6\pi t)
	\end{bmatrix},\quad 
	t \in [0,1].
	\]
	As the process noise $w$, we consider a
	zero-mean Gaussian white-noise process with covariance matrix 
	\[
	\mathrm{E}
	\left [w(t)w(s)^{\top} \right ] = \delta(t-s)10^{-1}I_4,
	\]
	where 
	$\mathrm{E}[\cdot]$ is the expectation operator 
	and $\delta$ is the Dirac
	delta function. Applying Lemma~\ref{lem:adjoint}, 
	we find that
	the squared norm of the synthesis operator associated with the
	generated noise
	is approximately $0.0905$.
	Theorem~\ref{thm:noisy_case} shows that 
	the state-input data are informative for 
	quadratic stabilization under the noise class $\Delta_{\tau,c}$
	with $c = 0.131$, where
	the LMIs are solved using
	MATLAB R2026a with YALMIP~\cite{Lofberg2004} and 
	MOSEK~\cite{MOSEK2026}.
	For $c = 0.1$, the stabilizing gain
	is obtained from the solutions 
	$P$ and $L$ of the LMIs as
	\[
	K = LP^{-1} = 
	\begin{bmatrix}
		0.4586 &  -5.1601  &  4.6073 &  -5.3129 \\
		9.1309 &   0.2532  &  7.2420  & -8.7416
	\end{bmatrix}.
	\]
	
	\section{Conclusion}
	\label{sec:conclusion}
	We developed a synthesis-operator
	framework for derivative-free data-driven
	control of continuous-time systems.
	First, we characterized the set of 
	data-consistent systems
	by embedding state and input trajectories
	into synthesis operators.
	Next,
	the finite-rank structure of the synthesis operators allowed us to obtain matrix-based characterizations of informativity properties for identification and for stabilization in the noise-free setting.
	Finally, we derived a necessary and sufficient
	LMI condition for noisy data to
	be informative for quadratic stabilization.
	Future work will extend the synthesis-operator approach to settings involving only input-output data or continuous-time data reconstructed from sampled measurements.

	\appendix
	\renewcommand{\thetheorem}{\Alph{theorem}}
	\setcounter{theorem}{0}
	
	\subsection{Norms of synthesis operators}
	\label{sec:norm_adjoint}
	In Section~\ref{sec:noisy_case},
	we have considered process noise whose associated synthesis operator $W$ satisfies $\|W\| \leq \sqrt{c}$ for some $c \geq 0$.
	To interpret this constraint,
	we examine the norm of a synthesis operator.
	Intuitively, the argument below shows that 
	$\|W\|$ is governed by
	the low-frequency components of the noise.
	In other words, 
	$\|W\|$ is small if the noise energy
	is concentrated in a high-frequency range.
	
	Let $F$ be the synthesis operator associated with
	$f \in \Ls([0,\tau]; \mathbb{R}^n)$.
	Since $\|F\| = \|F^*\|$, we turn our attention to the adjoint
	$F^*$.
	Consider $h \coloneqq F^*v$ for any $v \in \mathbb{R}^n$.
	By
	\eqref{eq:F_ad_with_C}, 
	$h$ is the solution of the boundary value
	problem 
	\begin{equation}
		\label{eq:boundary_value_problem}
		h''(t) = - f(t)^{\top} v,~~~t \in (0,\tau);\quad 
		h(0) = 0\text{~and~}h(\tau) = 0.
	\end{equation}
	From this observation, we can regard
	$\|F^*v\|_{\Hs} = \|h'\|_{\Ls}$ as the (potential) energy
	of the response $h$ driven by $- f(\cdot)^{\top} v$.
	
	To provide a frequency-domain interpretation of $\|F^*\|$, 
	define the orthonormal basis $(\psi_j)_{j \in \mathbb{N}}$ of $\Ls[0,\tau]$ by
	\[
	\psi_j(t) \coloneqq \sqrt{\frac{2}{\tau}} \sin\left(
	\frac{j \pi }{\tau} t
	\right)
	\]
	for $t \in [0,\tau]$ and $j \in \mathbb{N}$.
	Let $f_k$ denote the $k$-th element of $f$.
	For $N \in \mathbb{N}$, define
	the matrix
	$\Theta_N \in \mathbb{R}^{N \times n}$ by
	\begin{equation}
		\label{eq:Theta_N_def}
		\Theta_N \coloneqq \left[
		\frac{\langle f_k, \psi_{j} \rangle_{\Ls}}{j}
		\right]_{1 \leq j \leq N,\, 1 \leq k \leq n}.
	\end{equation}
	Here, the factor $1/j$ acts as a weight that scales
	the $j$-th row of $\Theta_N$.
	In the scalar case $n=1$, we have
	\[
	\|\Theta_N\| = \sqrt{\sum_{j=1}^N \left| \frac{\langle f, \psi_{j} \rangle_{\Ls}}{j}\right|^2}.
	\]
	Since $\langle f, \psi_{j} \rangle_{\Ls}$ is 
	the Fourier coefficient of $f$ associated with the frequency $j\pi/\tau$, the weight $1/j$
	implies that
	low-frequency
	components of $f$ contribute to
	$\|\Theta_N\|$ more than high-frequency ones.
	The following result shows that 
	$\|F\|$ is given by the limit of $\tau\|\Theta_N\|/\pi $ as $N\to \infty$.
	\begin{proposition}
		\label{prop:norms_adjoints}
		Let $F$ be the synthesis operator associated with
		$f \in \Ls([0,\tau]; \mathbb{R}^n)$. For $N \in \mathbb{N}$, define 
		the matrix $\Theta_N \in \mathbb{R}^{N\times n}$ by \eqref{eq:Theta_N_def}. Then
		\begin{equation}
			\label{eq:adjoint_norm}
			\|F\| =  \frac{\tau }{\pi}\lim_{N \to \infty}
			\|\Theta_N\|.
		\end{equation}
	\end{proposition}
	\begin{proof}
		Let $v \in \mathbb{R}^n$ be arbitrary, and let 
		$\ell^2(\mathbb{N})$ 
		be the space of square-summable sequences
		of real numbers, equipped with the standard inner product $\langle \cdot ,\cdot \rangle_{\ell^2} $. A simple calculation shows that
		the solution $h$ of the boundary value problem
		\eqref{eq:boundary_value_problem} is given by
		\[
		h(t) = \frac{\tau^2}{\pi^2}
		\sum_{j=1}^{\infty} \frac{\langle 
			f(\cdot)^{\top} v, \psi_{j }
			\rangle_{\Ls} }{j^2} \psi_{j}(t)
		\]
		for $t \in [0,\tau]$.
		Define
		$\Theta \colon \mathbb{R}^n \to \ell^2(\mathbb{N})$ by
		\[
		\Theta v \coloneqq 
		\left(
		\sum_{k=1}^n \frac{\langle f_k, \psi_{j} \rangle_{\Ls}}{j} 
		v_k
		\right)_{j \in \mathbb{N}}
		\]
		for 
		$
		v = \begin{bmatrix}
			v_1 & \cdots & v_n
		\end{bmatrix}^{\top} \in \mathbb{R}^n$.
		By the orthogonality of
		the cosine functions, it follows that for all $v \in \mathbb{R}^n$,
		\begin{equation}
			\label{eq:Theta_norm}
			\|F^*v\|_{\Hs} =
			\|h'\|_{\Ls} =
			\frac{\tau}{\pi}\|\Theta v\|_{\ell^2}.
		\end{equation}
		For $N \in \mathbb{N}$,
		define the truncation operator $\Pi_N$  on $\ell^2(\mathbb{N})$ by
		\[
		\Pi_N \eta \coloneqq 
		(\eta_1,\eta_2,\dots,\eta_N,0,0,\cdots),\quad 
		\eta = (\eta_j)_{j\in\mathbb{N}} \in \ell^2(\mathbb{N}).
		\]
		Since $\Theta$ is a finite-rank operator,
		it follows that $\|\Theta - \Pi_N \Theta\| \to 0$
		as $N \to \infty$. This and
		\eqref{eq:Theta_norm} yield \eqref{eq:adjoint_norm},
		since $\|F\| = \|F^*\|$.
	\end{proof}

\printbibliography
\end{document}